\documentclass[12pt]{amsart}
\usepackage[english]{babel}
\usepackage{amsmath}
\usepackage{amsthm}
\usepackage{amssymb}
\usepackage{mathrsfs}
\usepackage{enumerate}
\usepackage[notcite, final, notref]{showkeys}
\usepackage{dsfont}
\usepackage{color}

\newtheorem{theorem}{Theorem}[section]

\newtheorem{lemma}[theorem]{Lemma}
\newtheorem{proposition}[theorem]{Proposition}
\newtheorem{corollary}[theorem]{Corollary}
\newtheorem{definition}[theorem]{Definition}

\theoremstyle{definition}
\newtheorem{remark}[theorem]{Remark}

\title[Quaternionic positive real lemma]{Positive and generalized positive real lemma for slice hyperholomorphic functions}

\author[D. Alpay]{Daniel Alpay}
\address{(DA)
Faculty of Mathematics, Physics, and Computation\\
Schmid College of Science and Technology\\
Chapman University\\
One University Drive
Orange, California 92866\\
USA}
\email{alpay@chapman.edu}

\author[F. Colombo]{Fabrizio Colombo}
\address{(FC) Politecnico di
Milano\\Dipartimento di Matematica\\Via E. Bonardi, 9\\20133 Milano,
Italy}
\email{fabrizio.colombo@polimi.it}

\author[I. Lewkowicz]{Izchak Lewkowicz}
\address{(IL) Department of electrical engineering
Ben-Gurion University of the Negev\\ P.O.B. 653\\ Beer-Sheva, 84105\\
Israel}
\email{izchak@ee.bgu.ac.il}

\author[I. Sabadini]{Irene Sabadini}
\address{(IS) Politecnico di
Milano\\Dipartimento di Matematica\\Via E. Bonardi, 9\\20133 Milano\\Italy}
\email{irene.sabadini@polimi.it}

\begin{document}
\maketitle
\begin{abstract}
In this paper we prove a quaternionic positive real lemma as well as its generalized version, in case the associated kernel has negative squares for slice hyperholomorphic functions. We consider the case of functions with positive real part in the half space of quaternions with positive real part, as well as the case of (generalized) Schur functions in the open unit ball.
\end{abstract}

\noindent AMS Classification: 30C40, 30G35, 93B15.

\noindent {\em }
\date{today}
\tableofcontents
\section{Introduction}
\setcounter{equation}{0}
Scalar rational function which analytically map the open right half of the complex plane to its closure are called
in system and control engineering positive real are called $\mathcal{PR}$. More generally, matrix-valued functions analytic
in the open unit disk and having a positive real part there play an important role in engineering, network theory and operator theory.
They can be characterized in a number of ways, and we mention in particular Herglotz integral representations. They admit numerous generalizations, for instance to quaternions
(see e.g. \cite{MR330390911,MR3250507}) and to the case of negative squares. The latter was initiated and studied in much details by Krein and Langer; see e.g.
\cite{kl1,MR47:7504}.\smallskip

 The scalar rational case is of special interest. Indeed,
for nearly a century it has been recognized that the driving point impedance of an electrical circuit comprised of resistors, inductors and capacitors is a
$\mathcal{PR}$ function and conversely every $\mathcal{PR}$ function may be realized as a  driving point impedance of such a circuit; see \cite{brune-31,Ca1}.
A classical electro-mechanical duality reveals that mechanical systems comprised of mass-damper-spring  (strictly speaking, mass should be replaced by inerter;
see \cite{inerter} for the latter)
are described by  $\mathcal{PR}$ functions. Thus positive real functions are perceived as a mathematical model of linear passive systems.
The positive real  lemma (also called Kalman-Yakubovich-Popov lemma) offers an easy-to-check characterization of matrix-valued  $\mathcal{PR}$ functions with no pole at infinity.
There are a few hundreds of papers on this subject. The purpose of this work is to introduce an extension of
the positive real lemma in the setting of quaternionic analysis and rational slice hyperholomorphic functions.
To set the framework, we first discuss the complex setting, and recall that an important question in the theory of matrix-valued rational functions is to relate metric
properties of the function to those of one of its minimal realization. The case of
functions taking unitary values on the boundary is considered in \cite{ag,MR699564}. The positive real lemma characterizes minimal realizations of functions having a positive real part in a half-plane and reads (see \cite{Anderson_Moore,DDGK,faurre,MR525380}):

\begin{theorem}
\label{thmkyp1}
Let $R$ be a $\mathbb C^{n\times n}$-valued rational function, analytic at infinity, and with minimal realization
\begin{equation}
R(z)=D+C(zI_N-A)^{-1}B,
\end{equation}
with $D=R(\infty)\in\mathbb C^{n\times n}$ and $(C,A,B)\in\mathbb C^{n\times N}\times\mathbb C^{N\times N}\times\mathbb C^{N\times n}$.
Then $R$ has a real positive real part in the open right half-plane if and only if there exists
a negative definite matrix $H\in\mathbb C^{N\times N}$ satisfying
\begin{equation}
\label{conv1}
\begin{pmatrix}H&0\\0&I_n\end{pmatrix}\begin{pmatrix}A&B\\C&D\end{pmatrix}+\begin{pmatrix}A&B\\ C&D\end{pmatrix}^*\begin{pmatrix}H&0\\0&I_n\end{pmatrix}\ge 0.
\end{equation}
\end{theorem}

\begin{remark}{\rm The matrix $H$  is not unique and the set of all matrices $H$ satisfying \eqref{conv1} is convex. Assuming $D+D^*$ positive definite, the Schur complement formula shows that
\eqref{conv1} holds if and only if $H$ satisfies the Riccati inequality
\begin{equation}
HA+A^*H-(HB+C^*)(D+D^*)^{-1}(C+B^*H)\le 0,
\end{equation}
see \cite[Propositions 3.1 and 3.2]{AlpLew2011},
and it can be proved (see \cite{faurre}) that this set has a minimal and a maximal solution. The minimal solution correspond to the spectral factorization.
To see the link with factorizations, write the left side of \eqref{conv1} as
\[
\begin{pmatrix}L \\ M\end{pmatrix}\begin{pmatrix} L^*&M^*\end{pmatrix},
\]
with $L\in\mathbb C^{N\times \ell}$ and $M\in\mathbb C^{n\times \ell}$, and
where $\ell$ is equal to the rank of the left hand side of \eqref{conv1}.\smallskip

We have for $z\in\mathbb C$ (see \cite[p. 26]{faurre})
\begin{equation}
\label{facto}
R(z)+R(-\overline{z})^*=(M+C(zI-A)^{-1}H^{-1}L)(M+C(-\overline{z}I-A)^{-1}H^{-1}L)^*
\end{equation}
and so, for $x\in\mathbb R$
\begin{equation}
\label{facto1}
R(x)+R(-x)^*=(M+C(xI-A)^{-1}H^{-1}L)(M+C(-xI-A)^{-1}H^{-1}L)^*
\end{equation}
Note that \eqref{facto1} (and not \eqref{facto} because of the noncommutativity) makes sense in the quaternionic setting, and
then can be extended from $x\in\mathbb R$ to general quaternions $q\in\mathbb H$ using the $\star$-product  see formulas \eqref{thilde} and \eqref{thilde1} and next section.}
\end{remark}

\begin{remark}
{\rm Theorem \ref{thmkyp1} has been extended in \cite{DDGK} to the case of rational
functions which are only positive on the imaginary axis. The matrix $H$ in \eqref{conv1} is
then only Hermitian non-singular.}
\end{remark}

Closely related to the above right-half plane Kalman-Yakubovich-Popov Lemma in \eqref{conv1} (and its generalization to generalized positive functions)
is the analogous result for rational matrix-valued functions analytic in the open unit disk
and contractive on the unit circle (i.e. Schur functions), or possibly having a (finite) number of poles in the open unit disk (i.e.
generalized Schur functions). See \cite[(16), p. 604]{DDGK} for the following theorem.


\begin{theorem}
\label{gs1234}
Let $S$ be a ${\mathbb C}^{n\times m}$-valued rational function, analytic at infinity,
with minimal realization $S(z)=D+C(zI_N-A)^{-1}B$. Then, $S$ takes contractive
values on the unit circle if and only if there is an invertible Hermitian matrix $H$
such that
\begin{equation}
\label{opera4!4!6!7!8!8!21!}
\begin{pmatrix}H&0\\0&I_{n}\end{pmatrix}-
\begin{pmatrix}A&B\\ C&D\end{pmatrix}^*
\begin{pmatrix}H&0\\0&I_{n}\end{pmatrix}
\begin{pmatrix}A&B\\ C&D\end{pmatrix}
\ge 0.
\end{equation}
Furthermore, $S$ is a Schur function if and only if $H$ is negative.
\end{theorem}

Generalized Schur functions were characterized in \cite{kl1} as quotient
of Schur functions, the denominator being a finite matrix Blaschke product. This result allows
to reduce the study of realizations of generalized Schur functions to the study of
Schur functions. They are been studied in a number of works; see e.g. \cite{adrs},
\cite{MR2392765}, \cite{MR2002664}, \cite{MR1677962}, \cite[Section II]{vai}.\smallskip

When leaving the setting of the complex numbers and going to quaternionic slice hyperholomorphic functions, a number of questions have to be addressed, the first one being what are the objects
to be studied, that is, what are now Schur and generalized Schur functions, and positive real and generalized positive real Schur functions. Since the spectral theorem holds
for Hermitian matrices with quaternionic entries, one can defined positive definite kernels, and kernels having a finite number of negative squares; see \cite{zbMATH06658818} and Definition
\ref{neg123123} below. We will define the classes of functions to be studied in terms of the number of negative squares of an associated kernel; in the complex setting the two definitions
are equivalent, but this is not the case anymore in the quaternionic setting. The following definitions first appeared in \cite{acls_milan,MR3127378}. In the statement, and in this paper in general,
$\mathbb H$ denotes the skew-field of the quaternions.

\begin{definition}
A $\mathbb H^{n\times n}$-valued function $\phi$, rational slice hyperholomorphic in a neighborhood $\Omega$ of the origin in $\mathbb H$ is called positive
(resp. generalized positive)
if the kernel
\begin{equation}
K_\phi(p,q)=(\phi(p)+\phi(q)^*)\star (p+\overline{q})^{-\star}
\end{equation}
is positive definite (resp. has a finite number of negative squares) in $\Omega$.
\end{definition}

\begin{definition}
A $\mathbb H^{n\times m}$-valued function $S$ , rational slice hyperholomorphic in a neighborhood $\Omega$ of the origin in $\mathbb H$ is called a Schur function
(resp. a generalized Schur function) if the kernel
\begin{equation}
K_S(p,q)=\sum_{u=0}^\infty p^u(I_n-S(p)S(q)^*)\overline{q}^u=(I_n-S(p)S(q)^*) \star (1-p\overline{q})^{-\star}
\end{equation}
is positive definite (resp. has a finite number of negative squares) in $\Omega$.
\end{definition}
\begin{remark}{\rm
Realizations of operator-valued slice hyperholomorphic functions with real positive part have been considered in  \cite{acls_milan,zbMATH06658818}.
There the main tools were reproducing kernel Hilbert spaces and
the theory of relations in Hilbert and Pontryagin spaces. Here the strategy is different (although a bit overlapping).
The strategy used to prove a quaternionic positive real lemma is to start from a positive definite kernel in the quaternions, restrict it to the real line and apply the map $\chi$ (see
\eqref{chi123})
to obtain a condition on a complex matrix-valued positive definite on an open interval of the real line. An extension theorem
(called Fitzgerald's theorem in \cite[Theorem p. 144]{donoghue}; see \cite{zbMATH03282859})
allows to extend it to the complex plane and so apply the
classical positive real lemma on the new kernel.
The case of negative squares requires an extension of this result. The underlying structure allows to go back to the quaternionic setting.}
\end{remark}

\begin{remark}
 {\rm  When the kernel $K_\Phi$ is not assumed to have a finite number of negative squares, one can still obtain realizations with an interesting structure
using Krein spaces; see \cite{MR89a:47055,MR903068} and, in the quaternionic setting, \cite{2018arXiv180405255A}.}
\end{remark}

The paper contains eight sections besides the introduction, and we now proceed to describe their contents.
Section 2 contains some preliminaries on quaternions and slice hyperholomorphic functions. Section 3 provides some characterizations of the image of quaternionic matrices under a suitable map denoted by $\chi$ and also of the minimality of a triple of matrices.  In Section 4 we prove a quaternionic version of the positive real lemma in the case of the half space and unit
ball. Section 5
contains the extension result in the case of kernels with negative squares mentioned in the
previous paragraph. Finally, in Section 6 we extend the results in Section 4 to the case of negative squares.

\section{Slice hyperholomorphic functions}
\setcounter{equation}{0}
In this section we recall some basic facts about quaternions and slice hyperholomorphic functions which may be useful in the sequel. We refer the interested reader to the book \cite{zbMATH06658818} for a more complete treatment of these topics.
\\
The skew field $\mathbb{H}$ of quaternions consists of elements of the form
$p=x_0  + x_{1} i +x_{2} j + x_{3} k$, $x_i\in \mathbb{R}$, $i=0,1, 2,3$,
where the imaginary units $i, j, k$ satisfy
$$i^2=j^2=k^2=-1, \ ij=-ji=k, \ jk=-kj=i, \ ki=-ik=j.$$ The real number $x_0$ also denoted by ${\rm Re}(p)$ is called real part of $p$
while $ x_{1} i +x_{2} j + x_{3} k$ is called vector part or imaginary part of $p$.
A quaternion $p=x_0  + x_{1} i +x_{2} j + x_{3} k$ may be identified with a vector in $\mathbb R^4$ and its norm is the Euclidean norm in $\mathbb R^4$, i.e.
$|p|=\sqrt{x_0^2+x_1^2 +x_2^2+x_3^2}$. For any $p\in\mathbb H$ as above its conjugate is defined by $\overline p=x_0  - x_{1} i - x_{2} j - x_{3} k$; note that $p\overline p=\overline p p=|p|^2$.
\\
By selecting one of the imaginary units of the basis, e.g. the unit $i$, a quaternion $p$ can be written in terms of two complex numbers $z_1, z_2$ belonging to the complex plane containing elements of the form $x+iy$, namely $p =z_1+z_2j$ with $z_1=x_0+ix_1$, $z_2=x_2+ix_3$. Moreover, one may identify the quaternion $p$ with a $2\times 2$ complex valued matrix via the map $\chi_i$ defined by (see \cite{MR97h:15020})
\begin{equation}
\label{chi123}
\chi_i(p)=\begin{pmatrix}z_1&z_2\\ -\overline{z_2}&\overline{z_1}\end{pmatrix}.
\end{equation}
In the sequel, in order to ease the notation the subscript will be omitted and we will write $\chi$ instead of $\chi_i$. The map $\chi\, : \mathbb H\to\mathbb C^{2\times 2}$ is an injective homomorphism of rings, i.e. $\chi(p+q)=\chi(p)+\chi(q)$ and $\chi(pq)=\chi(p)\chi(q)$, for any $p,q\in\mathbb H$. Further properties of $\chi$ may be found in \cite{MR97h:15020}.
\\
It is also useful to introduce the set $\mathbb S$ of purely imaginary quaternions with square equal $-1$. This set coincides with the set of purely imaginary quaternions with norm $1$ and it is a $2$-dimensional sphere in $\mathbb H$ identified with $\mathbb R^4$. Any nonreal quaternion $q$ uniquely identifies the element $I\in\mathbb S$ defined by $I= x_{1} i +x_{2} j + x_{3} k/|x_{1} i +x_{2} j + x_{3} k|$, so that $p=x+I y$ with $x=x_0$ and $y=|x_{1} i +x_{2} j + x_{3} k|$. Moreover $p$ defines a $2$-sphere $[p]$ given by $[p]=\{x+Jy, \ J\in\mathbb S\}$.
\\
We now introduce the notion of slice hyperholomorphic function with values in a quaternionic Banach space. These functions are naturally defined on open sets $\Omega$ which are axially symmetric, i.e. such that $p\in\Omega$ implies that $[p]\subset\Omega$. It is immediate that the quaternionic unit ball $\mathbb B$ and the half space $\mathbb H_+$ of quaternions with positive real part are axially symmetric.
\begin{definition}
Let $\Omega\subseteq \mathbb H$ be an axially symmetric set and let $\mathcal{X}$ be a two sided quaternionic Banach space. A function
$f:\ \Omega\to\mathcal{X}$ of the form
$f(p)=f(x+{I}y)=\alpha (x,y) +{I}\beta (x,y)$ where $\alpha, \beta:
\Omega\to \mathcal{X}$ depend only on $x,y$, are real differentiable, satisfy the Cauchy-Riemann
equations
\begin{equation}\label{CR1}
\begin{cases}
\partial_x \alpha -\partial_y\beta=0\\
\partial_y \alpha
+\partial_x\beta=0,
\end{cases}
\end{equation}
and
\begin{equation}\label{alfabeta1}
\begin{split}
\alpha(x,-y)=\alpha(x,y),\qquad
\beta(x,-y)=-\beta(x,y).
\end{split}
\end{equation}
is said to be (left) slice hyperholomorphic.
\end{definition}
Observe that, in the definition, if $p=x$ is a real
quaternion, then ${I}$ is not uniquely defined but the hypothesis
that $\beta$ is odd in the variable $y$ implies $\beta(x,0)=0$.
\\
One can also introduce the definition of right slice hyperholomorhic function: th eonly cahnge in the previous definition is that $f(p)=f(x+{I}y)=\alpha (x,y) +\beta (x,y){I}$.
\\
The definition holds, in particular, when $\mathcal{X}=\mathbb H^{n\times m}$ the quaternionic two sided Banach space of $n\times m$ matrices with quaternionic entries and in the scalar valued case.
\\
When we consider two slice hyperholomorhic functions $f,g$ on $\Omega$ with values in $\mathcal X$, $\mathcal X_1$, respectively, and it makes sense to define the product $AB$ of $A\in\mathcal X$, $B\in\mathcal X_1$ we can define the $\star$-product of $f$ and $g$ as follows: let $f(p)=f(x+{I}y)=\alpha (x,y) +{I}\beta (x,y)$, $g(p)=g(x+{I}y)=\gamma (x,y) +{I}\delta (x,y)$, then
$$
f\star g= (\alpha\gamma -\beta \delta)+I(\alpha \delta+\beta \gamma).
$$
In the case a function $f$ is scalar valued, it makes sense to define its $\star$-reciprocal. We refer the reader to \cite{MR2752913,zbMATH06658818} for more details. Here it is sufficient to recall how to construct the $\star$-inverses of the functions $\varphi(p)=p+q$, $\psi(p)=1-p\overline q$ which are left hyperholomorphic in $p$ (and right slice hyperholomorphic in $\overline q$):
\[
\begin{split}
(1-p\overline q)^{-\star}&=(1-2{\rm Re}(q)p+|q|^2p^2)^{-1}(1-pq)\\
(p+\overline q)^{-\star}&=(|q|^2+2{\rm Re}(q)p+p^2)^{-1}(p+q).
\end{split}
\]
\section{Minimality}
\setcounter{equation}{0}
The map $\chi=\chi_i$ that has been introduced in the previous section can be extended to quaternionic matrices, in fact if $M=(q_{\ell s})\in\mathbb H^{n\times m}$, $q_{\ell s}=z_{1,\ell s}+ z_{2,\ell s}j$, $z_{1,\ell s}$, $z_{2,\ell s}\in\mathbb C$ we set $M=A+Bj$ where $A=(z_{1,\ell s})$, $B=(z_{2,\ell s})\in\mathbb C^{n\times m}$. Then we define
\[
\chi(M)=\chi(A+Bj)=
\begin{pmatrix}
A&B\\
-\overline{B}&\overline{A}\end{pmatrix}.
\]
In this section, we study the connections between the matrices associated to a quaternionic realization and their images under the map $\chi$. Our first goal is to characterize the range of $\chi$.
\begin{proposition}
Let $(A,B)\in\mathbb C^{n\times m}\times\mathbb C^{n\times m}$ and let $\xi_1,\xi_2\in\mathbb C^m$ be such that
\begin{equation}
\label{bastille}
\begin{pmatrix}
A&B\\
-\overline{B}&\overline{A}\end{pmatrix}
\begin{pmatrix}\xi_1\\ -\overline{\xi_2}\end{pmatrix}=
\begin{pmatrix}0_n\\ 0_n\end{pmatrix}.
\end{equation}
Then
\begin{equation}
\label{bastille4}
\begin{pmatrix}
A&B\\
-\overline{B}&\overline{A}\end{pmatrix}\begin{pmatrix}\xi_2\\ \overline{\xi_1}\end{pmatrix}=
\begin{pmatrix}0_n\\ 0_n\end{pmatrix}.
\end{equation}
\end{proposition}

\begin{proof}
Applying conjugation to both sides of \eqref{bastille} we get
\begin{equation}
\label{bastille1}
\begin{pmatrix}
\overline{A}&\overline{B}\\
-{B}&{A}\end{pmatrix}\begin{pmatrix}\overline{\xi_1}\\ -\xi_2\end{pmatrix}=
\begin{pmatrix}0_n\\ 0_n\end{pmatrix},
\end{equation}
and so
\begin{equation}
\label{bastille11}
\underbrace{\begin{pmatrix}0&-I_m\\ I_m&0\end{pmatrix}
\begin{pmatrix}
\overline{A}&\overline{B}\\
-{B}&{A}\end{pmatrix}
\begin{pmatrix}0&I_m\\ -I_m&0\end{pmatrix}}_{\begin{pmatrix}
A&B\\
-\overline{B}&\overline{A}\end{pmatrix}}
\begin{pmatrix}0&-I_m\\ I_m&0\end{pmatrix}
\begin{pmatrix}\overline{\xi_1}\\ -\xi_2\end{pmatrix}=
\begin{pmatrix}0_n\\ 0_n\end{pmatrix},
\end{equation}
which is \eqref{bastille4}.
\end{proof}

As a corollary we get:

\begin{corollary}
With $A,B,\xi_1$ and $\xi_2$ as above,
\begin{equation}
\label{voltaire!2!2!2}
\chi(A+Bj)\begin{pmatrix}\overline{\xi_1}\\ -\xi_2\end{pmatrix}=
\begin{pmatrix}0_n\\ 0_n\end{pmatrix}\quad\iff\quad (A+Bj)(\xi_1+\xi_2j)=0.
\end{equation}
\end{corollary}

\begin{lemma}
\label{oufouf}
A matrix $M\in(\mathbb C^{n\times m})^{2\times 2}$ is of the form
\[
\begin{pmatrix}
A&B\\
-\overline{B}&\overline{A}\end{pmatrix}
\]
if and only if it satisfies
\begin{equation}
\label{opera4!4!4!4!}
E_n^{-1}\overline{M}E_m=M,
\end{equation}
where
\begin{equation}
E_m=\begin{pmatrix}0&I_m\\ -I_m&0\end{pmatrix}.
\end{equation}
\end{lemma}

\begin{proof}
Let
\[
M=\begin{pmatrix}M_{11}&M_{12}\\ M_{21}& M_{22}\end{pmatrix}
\]
be the decomposition of $M$ into $\mathbb C^{n\times m}$ matrices. We have
\[
E_n^{-1}\overline{M}E_m=\begin{pmatrix}\overline{M_{22}}&-\overline{M_{21}}\\ -\overline{M_{12}}&\overline{M_{11}}\end{pmatrix},
\]
which will be equal to $M$ if and only if $M_{11}=\overline{M_{22}}$ and $M_{12}=-\overline{M_{21}}$, that is, if and only if $M$ is of the form \eqref{opera4!4!4!4!}.
\end{proof}
\begin{remark}\label{REMp}{\rm
We note that
\begin{equation}
E_m^{-1}=-E_m=E_m^*
\label{ouf}
\end{equation}
and that the map $a$ defined by
\begin{equation}
\label{a}
a(M)= E_m^{-1}\overline{M}E_m
\end{equation}
is an involution and is real-linear. Moreover, $a$ maps positive matrices into positive matrices since $E_m^{-1}=E_m^*$.}
\end{remark}
The notion of minimality in this section is according to the following definition:
\begin{definition}
We say that the triple of matrices $(C,A,B)\in\mathbb H^{n\times
N}\times\mathbb H^{N\times N}\times\mathbb H^{N\times m}$ is
minimal if the following two conditions hold: The pair $(C,A)$ is
observable, that is
\begin{equation}
\label{ca1}
\bigcap\limits_{t=0}^\infty\ker CA^t=\left\{0\right\}
\end{equation}
and the pair $(A,B)$ is controllable, that is
\begin{equation}
\bigcap\limits_{t=0}^\infty\ker B^*A^{*t}=\left\{0\right\}.
\label{ab1}
\end{equation}
\end{definition}
Note that these intersections are finite (see \cite{zbMATH06658818}, Section 9).
The map $\chi$ allows to reduce the study of minimality to the case of matrices with
complex entries.

\begin{theorem}
\label{thm146}
Let $(C,A,B)\in\mathbb H^{n\times
N}\times\mathbb H^{N\times N}\times\mathbb H^{N\times m}$ .\\
$(1)$ The pair $(C,A)\in\mathbb H^{n\times
N}\times\mathbb H^{N\times N}$ is controllable if and only if the corresponding pair
$(\chi(C),\chi(A)$ is controllable.\\
$(2)$ The pair $(A,B)\in\mathbb H^{N\times
N}\times\mathbb H^{N\times n}$ is observable if and only if the corresponding pair
$(\chi(A),\chi(B)$ is controllable.\\
$(3)$ The triple $(C,A,B)\in\mathbb H^{n\times
N}\times\mathbb H^{N\times N}\times\mathbb H^{N\times m}$
is minimal if and only if the corresponding
triple $(\chi(C),\chi(A),\chi(B))$ is minimal.
\end{theorem}

\begin{proof}
Using \eqref{voltaire!2!2!2} we have
\[
\xi_1+\xi_2 j\,\in\, \cap_{u=0}^\infty \ker CA^u\,\,\iff\,\, \begin{pmatrix}\xi_1\\ -\overline{\xi_2}
\end{pmatrix}\,\in\,\cap_{u=0}^\infty\ker \chi(C)\chi(A)^u,
\]
from which we get the first claim; the other claims follow easily.
\end{proof}

We remark that other kind of realizations are possible. Assuming that $R$ is slice hyperholomorphic at $\infty$ (that is, in a set of the form $|p|>r$)
one can consider realizations of the form
\[
R(p)=D+C\star (pI-A)^{-\star}B.
\]

In the setting of complex numbers and analytic functions, a matrix-valued analytic function is contractive in the open unit disk if and only if the kernel
\begin{equation}
\label{ks}
K_S(z,w)=\frac{I_n-S(z)S(w)^*}{1-z\overline{w}}
\end{equation}
is positive definite in the open unit disk. {
We now mention an extension result (see \cite[Th\'eor\`eme 2.6.5, p. 45]{MR1638044}, which roughly
speaking, states that, in the positive case, positivity of the kernel $K_S$ implies analyticity
of the function $S$. Such results (for Carath\'eodory functions rather than Schur functions)
originate in the work of Loewner, see \cite{donoghue,Loewner}, when the positivity is assumed on
the boundary.
Here we use a different result , where the positivity is not on the boundary. For completeness we give a proof.

\begin{theorem}
\label{thpanorama}
Let $S$ be a $\mathbb C^{n\times m}$-valued function defined on a subset $\Omega\subset\mathbb D$
and such that the kernel $K_S$ is positive definite in $\Omega$. Assume that $\Omega$ has an
accumulation point inside $\mathbb D$. Then, $S$ is the restriction of a uniquely defined Schur
function in $\mathbb D$, and conversely.
\end{theorem}

\begin{proof}[Outline of the proof] To ease the notation we consider the scalar case and write $s(z)$ rather than $S(z)$.
Only one direction is non trivial and will be proved. Let $\mathbf H_2$ denote the Hardy space of
the disk. Define
a linear relation $\mathcal R$ in $\mathbf H_2\times \mathbf H_2$ as the linear span of the pairs
\[
\left(\frac{1}{1-z\overline{w}},\frac{\overline{s(w)}}{1-z\overline{w}}\right),\quad w\in\Omega.
\]
$\mathcal R$ is densely defined since $\Omega$ has an accumulation point, and is contractive
because the kernel $K_S$ is positive definite on $\Omega$. It follows that $\mathcal R$ extends
to the graph of an everywhere defined contraction, say $T$. Its adjoint satisfies
\[
(T^*1)(w)=s(w),\quad w\in\Omega
\]
and is analytic, and hence $s$ extends to a uniquely defined analytic function. To see that this
function is a contraction, first check that for every integer $n$, and with $M_z$ denoting the
operator of multiplication by $z$,
\[
(T^*M_z^n1)(w)=w^ns(w),\quad w\in \Omega
\]
Thus $T^*$ is the multiplication operator by $s$, and so $s$ is contractive since $T^*$ is
contractive.
\end{proof}

\begin{remark}{\rm An alternative proof consists of solving the Nevanlinna-Pick theorem for an increasing family of finite subsets of $\Omega$ and use Montel's theorem, see also
Corollary 3.4 in \cite{2013arXiv1308.2658A} for the quaternionic setting.
An earlier result, also for the quaternionic setting and for a subclass of Schur
functions, is given in \cite[Theorem 4.9, p. 98]{MR3127378}.}\end{remark}
}
In the quaternionic setting, the kernel $K_S$ (defined by \eqref{ks}) takes now the form
\begin{equation}
\label{kh123}
\left(I_n-S(p)S(q)^*\right)\star(1-p\overline{q})^{-\star}.
\end{equation}
The equivalence of the positivity of \eqref{kh123} with the fact that $S$ is slice hyperholomorphic in the  quaternionic unit ball and bounded by $1$ in norm there
holds for $\mathbb H$-valued functions (see \cite{2013arXiv1308.2658A}),
but not for matrix-valued functions, as illustrated by the example
(see \cite[\S 8.4, p. 213]{zbMATH06658818})
\[
U(p)=\frac{1}{\sqrt 2}\begin{pmatrix}p&i\\ pi&1\end{pmatrix}.
\]
By Cayley transform a similar remark holds for Herglotz functions.
\section{The quaternionic positive real lemma}
\setcounter{equation}{0}

We now give the quaternionic version of the positive real lemma. As in the case of complex numbers, the set of Hermitian  matrices $H$ appearing in the statement
form a convex set.

\begin{theorem}
\label{kypbastille}
Let $\phi$ be a $\mathbb H^{n\times n}$-valued slice hyperholomorphic rational function,
slice hyperholomorphic at infinity, with minimal realization
\begin{equation}
\label{realinfi}
\phi(p)=D+C\star(pI_N-A)^{-\star} B
\end{equation}
Then,
\begin{equation}
\label{kphi}
(\phi(p)+\phi(q)^*)\star(p+\overline{q})^{-\star}\ge 0,\quad p,q\in\mathbb H_+
\end{equation}
if and only if there exists a negative definite
matrix $H\in\mathbb H^{N\times N}$ such that
\begin{equation}
\label{opera4!4!6!7!}
\begin{pmatrix}H&0\\0&I_{n}\end{pmatrix}
\begin{pmatrix}A&B\\ C&D\end{pmatrix}+
\begin{pmatrix}A&B\\ C&D\end{pmatrix}^*
\begin{pmatrix}H&0\\0&I_{n}\end{pmatrix}\ge 0.
\end{equation}
\end{theorem}

\begin{proof}
Restricting to the positive real line and using the map $\chi$, we get that the
kernel
\begin{equation}
\label{chiphi}
\frac{\chi(\phi)(x)+\chi(\phi)(y)^*}{x+y}
\end{equation}
is positive definite for $x,y$ on the positive real line.
By Theorem \ref{thm146} the realization
\[
\chi(\phi)(x)=\chi(D)+\chi(C)\left(xI_{2N}-\chi(A)\right)^{-1}\chi(B)
\]
is minimal. It follows from Theorem \ref{thpanorama} (using  Cayley transforms both on the variable and on $\chi(\phi)$) that
the kernel is positive definite in the whole right half-plane. We can therefore apply the
positive real lemma which asserts that there exits a negative definite matrix $H\in
\mathbb C^{2N\times 2N}$ such that

\begin{equation}
\label{opera4!4!}
\begin{pmatrix}H&0\\0&I_{2n}\end{pmatrix}
\begin{pmatrix}\chi(A)&\chi(B)\\ \chi(C)&\chi(D)\end{pmatrix}+
\begin{pmatrix}\chi(A)&\chi(B)\\ \chi(C)&\chi(D)\end{pmatrix}^*
\begin{pmatrix}H&0\\0&I_{2n}\end{pmatrix}\ge 0.
\end{equation}
We now show that there exists a solution $H_1$ to \eqref{opera4!4!}
such that $H_1=a(H_1)$, where $a$ is defined by \eqref{a}. By Lemma \ref{oufouf} and Remark \ref{REMp}, we deduce that there exists a positive quaternionic matrix $X$
such that $H_1=\chi(X)$.
To that purpose, we show that if $H$ satisfies \eqref{opera4!4!} so does $a(H)$. Then, $\frac{H+a(H)}{2}$ satisfies the symmetry \eqref{opera4!4!4!4!}.
Using \eqref{ouf} we first note that
\[
\begin{pmatrix}E_N^{-1}&0\\0&E_n^{-1}\end{pmatrix}=\begin{pmatrix}E_N&0\\0&E_n\end{pmatrix}^*.
\]
Thus, we have from
\eqref{opera4!4!}
\begin{equation}
\label{opera4!4!2!}
\begin{split}
\begin{pmatrix}E_N^{-1}&0\\0&E_n^{-1}\end{pmatrix}
\begin{pmatrix}\overline{H}&0\\0&I_{2n}\end{pmatrix}
\begin{pmatrix}E_N&0\\0&E_n\end{pmatrix}
\begin{pmatrix}E_N^{-1}&0\\0&E_n^{-1}\end{pmatrix}
\begin{pmatrix}\overline{\chi(A)}&\overline{\chi(B)}\\ \overline{\chi(C)}&\overline{\chi(D)}
\end{pmatrix}\begin{pmatrix}E_N&0\\0&E_n\end{pmatrix}+\\
&\hspace{-13.5cm}
+\begin{pmatrix}E^{-1}_N&0\\0&E^{-1}_n\end{pmatrix}
\begin{pmatrix}\overline{\chi(A)}&\overline{\chi(B)}\\ \overline{\chi(C)}&\overline{\chi(D)}
\end{pmatrix}^*
\begin{pmatrix}E_N&0\\0&E_n\end{pmatrix}\begin{pmatrix}E^{-1}_N&0\\0&E^{-1}_n\end{pmatrix}
\begin{pmatrix}\overline{H}&0\\0&I\end{pmatrix}\begin{pmatrix}E_N&0\\0&E_n\end{pmatrix}\ge 0.
\end{split}
\end{equation}
Using Lemma \ref{oufouf} we have

\begin{equation}
\label{opera4!4!2!1@}
\begin{split}
\begin{pmatrix}E_N^{-1}&0\\0&E_n^{-1}\end{pmatrix}
\begin{pmatrix}\overline{H}&0\\0&I_{2n}\end{pmatrix}
\begin{pmatrix}E_N&0\\0&E_n\end{pmatrix}
\begin{pmatrix}{\chi(A)}&{\chi(B)}\\ {\chi(C)}&{\chi(D)}\end{pmatrix}
+\\
&\hspace{-5cm}
+
\begin{pmatrix}{\chi(A)}&{\chi(B)}\\ {\chi(C)}&{\chi(D)}
\end{pmatrix}^*
\begin{pmatrix}E^{-1}_N&0\\0&E^{-1}_n\end{pmatrix}
\begin{pmatrix}\overline{H}&0\\0&I_n\end{pmatrix}\begin{pmatrix}E_N&0\\0&E_n\end{pmatrix}\ge 0.
\end{split}
\end{equation}

Thus $E^{-1}_N\overline{H}E_N$ satisfies \eqref{opera4!4!}, and so $H_1=\frac{H+a(H)}{2}$ is in the range of $\chi$.
By Lemma \ref{oufouf} with $H=\chi(X)$ it follows \eqref{opera4!4!6!7!}.\\

The converse follows the computations
in \cite[p. 26]{faurre} and in \cite[p. 129]{MR525380}. As in \cite{faurre} we denote by
\[
\begin{pmatrix}Q&S\\ S^*&R\end{pmatrix}\ge 0
\]
the right handside of \eqref{opera4!4!6!7!}. Then
\begin{align}
HA+A^*H&=Q\\
HB+C^*&=S\\
D+D^*&=R.
\end{align}

For $x,y\in\mathbb R$, and replacing $B$ by $H^{-1}(S-C^*)$ we have
\[
\begin{split}
\phi(x)+\phi(y)^*&=D+C(xI-A)^{-1}B+D^*+B^*(yI-A^*)^{-1}C^*\\
&=D+C(xI-A)^{-1}H^{-1}(S-C^*)+D^*+(S^*-C)H^{-1}(yI-A^*)^{-1}C^*\\
&=R+C(xI-A)^{-1}H^{-1}S+S^*H^{-1}(yI-A^*)^{-1}-\\
&\hspace{5mm}-
\left\{C(xI-A)^{-1}H^{-1}C^*+CH^{-1}(yI-A^*)^{-1}C^*\right\}\\
&=R+C(xI-A)^{-1}H^{-1}S+S^*H^{-1}(yI-A^*)^{-1}-\\
&\hspace{5mm}-(x+y)
\left\{C(xI-A)^{-1}H^{-1}(yI-A^*)^{-1}C^*\right\}+\\
&\hspace{5mm}+C(xI-A)^{-1}H^{-1}QH^{-1}(yI-A^*)^{-1}C^*.
\end{split}
\]
Hence
\begin{equation}
\label{thilde1}
\begin{split}
K_\phi(x,y)=
\frac{\phi(x)+\phi(y)^*}{x+y}&=-C(xI-A)^{-1}H^{-1}(yI-A^*)^{-1}C^*+\\
&\hspace{5mm}+\frac{
\begin{pmatrix}
C(xI-A)^{-1}H^{-1}&I\end{pmatrix}\begin{pmatrix}Q&S\\ S^*&R\end{pmatrix}
\begin{pmatrix}H^{-1}(yI-A^*)^{-1}C^*\\I\end{pmatrix}}{x+y}.
\end{split}
\end{equation}
The kernel \eqref{thilde1} admits an extension slice hyperholomorphic in $p$ and right slice hyperholomorphic in $\overline{q}$ given by
\begin{equation}\label{411primo}
\begin{split}
K_\phi(p,q)&=(\phi(p)+\phi(q)^*)(p+\overline{q})^{-\star}\\
&=-C\star (pI-A)^{-\star}H^{-1}(\overline{q}-A^*)^{-\star_r}\star_rC^*+\\
&\hspace{5mm}+\left(
\begin{pmatrix}
C\star (pI-A)^{-\star}H^{-1}&I\end{pmatrix}\begin{pmatrix}Q&S\\ S^*&R\end{pmatrix}
\begin{pmatrix}H^{-1}(\overline{q}I-A^*)^{-\star_r} \star_r C^*\\I\end{pmatrix}\right)\star (p+\overline{q})^{-\star}
\end{split}
\end{equation}
(where all the $\star$-product are computed in the variable $p$) and so $K_\phi$ is positive definite in the right half-space.

\end{proof}

\begin{remark}{\rm
Alternatively, and replacing now $C$ by $S^*-HB^*$ we also have
\[
\begin{split}
\phi(x)+\phi(y)^*&=D+C(xI-A)^{-1}B+D^*+B^*(yI-A^*)^{-1}C^*\\
&=R+(S-B^*H)(xI-A)^{-1}B+B^*(yI-A^*)^{-1}(S^*-HB)\\
&=R+S(xI-A)^{-1}B+B^*(yI-A)^{-*}S^*-\\
&\hspace{5mm}-\left\{B^*H(xI-A)^{-1}B+B^*(yI-A^*)^{-1}HB\right\}\\
&=R+S(xI-A)^{-1}B+B^*(yI-A)^{-*}S^*-\\
&\hspace{5mm}-\left\{B^*(yI-A^*)^{-1}\left(H(xI-A)+(yI-A^*)H\right)B+B^*(xI-A)^{-1}B\right\}\\
&=R+S(xI-A)^{-1}B+B^*(yI-A)^{-*}S^*-\\
&\hspace{5mm}-(x+y)B^*(yI-A^*)^{-1}H(xI-A)^{-1}B+\\
&\hspace{5mm}+B^*(yI-A^*)^{-1}Q(xI-A)^{-1}B.
\end{split}
\]
Hence
\begin{equation}
\label{thilde}
\begin{split}
\frac{\phi(x)+\phi(y)^*}{x+y}&=-B^*(yI-A^*)^{-1}H(xI-A)^{-1}B+\\
&\hspace{5mm}+\frac{
\begin{pmatrix}
B^*(yI-A^*)^{-1}&I\end{pmatrix}\begin{pmatrix}Q&S\\ S^*&R\end{pmatrix}
\begin{pmatrix}(xI-A)^{-1}B\\I\end{pmatrix}}{x+y}.
\end{split}
\end{equation}
}
\end{remark}
\mbox{}\\

We now consider the case of the quaternionic unit ball.

\begin{theorem}
A $\mathbb H^{n\times n}$-valued rational function slice hyperholomorhic at $\infty$ and with minimal realization \eqref{realinfi} is contractive in the quaternionic unit ball if and only if
there exists a negative definite matrix $H\in\mathbb H^{N\times N}$ such that
\begin{equation}
\label{opera4!4!6!7!8!!}
\begin{pmatrix}H&0\\0&I_{n}\end{pmatrix}-
\begin{pmatrix}A&B\\ C&D\end{pmatrix}^*
\begin{pmatrix}H&0\\0&I_{n}\end{pmatrix}
\begin{pmatrix}A&B\\ C&D\end{pmatrix}
\ge 0.
\end{equation}
\label{thmko}
\end{theorem}

\begin{proof}
As in the previous, one reduces the claim to the complex setting using the map $\chi$, and
use \cite[Theorem 2]{DDGK}. We then go back to the quaternionic setting as in the
proof of the previous  theorem.
\end{proof}

We also remark that the set of Hermitian matrices $H$ satisfying \eqref{opera4!4!6!7!8!!} form a convex set.

\section{The case of negative squares: an extension theorem}
\setcounter{equation}{0}
We start this section by recalling a definition which, as explained in the introduction, makes sense since the spectral theorem holds for quaternionic Hermitian matrices (see also \cite{zbMATH06658818}).

\begin{definition}
\label{neg123123}
Let $K(u,v)$ be a $\mathbb H^{n\times n}$-valued function defined for $u,v$ in some set $\Omega$ and assume that $K(u,v)=K(v,u)^*$ for all $u,v\in\Omega$.
The function $K(u,v)$ is said to have $\kappa$ negative squares if the following condition holds: for every $N\in\mathbb N$, and every choice of points $u_1,\ldots, u_N\in\Omega$ and
vectors $\xi_1,\ldots, \xi_N\in\mathbb H^n$ the $N\times N$  Hermitian matrix with $(k,j)$ entry $\xi_k^*K(u_k,u_j)\xi_j$ has at most $\kappa$ strictly negative eigenvalues, and exactly $\kappa$
strictly negative eigenvalues for some choice of $N,u_1,\ldots, u_N,\xi_1,\ldots, \xi_N$.
\end{definition}
\begin{theorem}
Let $F$ be a $\mathbb C^{n\times n}$-valued rational function, analytic at infinity. Then the following are equivalent:\smallskip

$(1)$ ${\rm Re}\, F\,(iy)\ge 0,\quad y\in\mathbb R$.\smallskip

$(2)$ The kernel
\[
K_F(z,w)=\frac{F(z)+F(w)^*}{z+\overline{w}}
\]
has a finite number of negative squares in the right half-plane.
\label{thm123}
\end{theorem}

\begin{proof}
Assume $(1)$. By maybe adding to $F$ a term of the form $i\epsilon I_n$ we do not change the first condition but may assume that the Cayley transform $S=(I_n+F)^{-1}(I_n-F)$ is well defined. We make
another Cayley transform this time on the variable, to get to a function $S_1$ contractive on the unit circle. By multiplying $S_1$ by a finite scalar Blaschke factor $b$ we get a function $bS_1$ which
is analytic in $\mathbb D$ and contractive on the unit circle. It is therefore a Schur function $S_0$ and so $S_1=\frac{S_0}{b}$ is a generalized Schur function, and the corresponding kernel
\[
K_{S_1}(z,w)=\frac{I_n-S(z)S(w)^*}{1-z\overline{w}}
\]
has a finite number of negative squares, see \cite{kl1}.
Since $K_F$ and $K_{S_1}$ are related by Cayley transforms, we get that $K_F$ has a finite number of negative squares.\smallskip

The converse is proved by reading backwards the above arguments.
\end{proof}

We now remark that Theorem \ref{thpanorama} will not hold in general in the case of negative squares. As recalled in \cite[p. 82]{adrs}, the function defined by
\[
S(z)=\begin{cases} 1,\quad z\in\mathbb D\setminus\left\{0\right\}\\
                0,\quad z=0,\end{cases}
\]
is not meromorphic in the open unit disk, while the corresponding kernel $K_S$ has one negative square. A claim on a counterpart of the extension theorem was made in
\cite[Theorem 3.2]{zbMATH01821264}, but the proof turned out to be flawed;
see \cite{erratumzbMATH01821264}. We now prove a weaker
result, which is enough for our purposes, when one knows ahead of time that the function $K_S$ is meromorphic in the open unit disk, and has
a finite number of negative squares in a subset $\Omega$ which has an accumulation point.
We focus on the matrix-valued rational case for simplicity, and which is  what is needed in
the present paper. The result itself can be seen as an extension of \cite[p. 144]{donoghue} to the setting of negative squares (although beyond the
scope of the current work, note that the case of analytic functions can be similarly treated).

\begin{theorem}
\label{th1234}
Let $S$ be a $\mathbb C^{n\times n}$-valued rational function such that the kernel $K_S(z,w)$ has
a finite number of negative squares in $(-1,1)$, from which have been removed possibly a finite
number of points $P$. Then, the kernel $K_S(z,w)$ has the same number of negative squares in the open
unit disk, from which have been removed possibly a finite number of points $P_1$.
\end{theorem}

Before proving the theorem we recall a result on positive definite kernels.

\begin{theorem}
Let $\Lambda_1(z,w)$ and $\Lambda_2(z,w)$ be two $\mathbb C^{n\times n}$-valued kernels, positive definite on the set $\Omega$, and let $\mathcal H(\Lambda_1),\mathcal H(\Lambda_2)$ and
$\mathcal H(\Lambda_1+\Lambda_2)$ denote the reproducing kernel Hilbert spaces associated to $\Lambda_1$, $\Lambda_2$ and $\Lambda_1+\Lambda_2$ respectively. Then
\begin{equation}
\mathcal H(\Lambda_1+\Lambda_2)=\mathcal H(\Lambda_1)+\mathcal H(\Lambda_2).
\end{equation}
More precisely, every element in $\mathcal H(\Lambda_1+\Lambda_2)$ can be written as a (in general non-unique) sum $f=f_1+f_2$, where $f_j\in\mathcal H(\Lambda_j)$, $j=1,2$. For every such decomposition
\begin{equation}
\label{ineqineq}
\|f\|^2_{\mathcal H(\Lambda_1+\Lambda_2)}\le \|f_1\|^2_{\mathcal H(\Lambda_1)}+\|f_2\|^2_{\mathcal H(\Lambda_2)},
\end{equation}
and there is a unique decomposition for which equality holds.
For the function
\begin{equation}
\label{bastille!!!!}
f(z)=\sum_{u=1}^U (\Lambda_1(z,w_u)+\Lambda_2(z,w_u))c_u,
\end{equation}
the minimal (unique) decomposition for which equality holds in \eqref{ineqineq} is given by
\begin{equation}
\label{f1f2!!!}
f_1(z)=\sum_{u=1}^U \Lambda_1(z,w_u)c_u,\quad and\quad f_2(z)=\sum_{u=1}^U \Lambda_2(z,w_u)c_u.
\end{equation}
\label{rosalind}
\end{theorem}

The proof of this result can be found in \cite{aron}. The fact that the equality in \eqref{ineqineq} holds for \eqref{f1f2!!!} when $f$ is of the form \eqref{bastille!!!!}
is crucial in the proof.

\begin{proof}[Proof of Theorem \ref{th1234}]
The strategy of the proof is as follows.
\[
K_S(z,w)=\frac{I_n}{1-z\overline{w}}-\frac{S(z)S(w)^*}{1-z\overline{w}}
\]
is a difference of two positive definite functions, and hence is the reproducing kernel of a
(in general not uniquely defined) reproducing kernel Krein space; see \cite{schwartz}. We study the restriction of the elements of
this space to $(-1,1)\setminus P$ to see that it is a Pontryagin space.\\

We proceed in a number of steps.\\

STEP 1: {\sl The function
\[
M_S(z,w)=\frac{I_n}{1-z\overline{w}}+\frac{S(z)S(w)^*}{1-z\overline{w}}
\]
is positive definite in $\mathbb D\setminus P_1$. It is the reproducing kernel of a (uniquely defined)
Hilbert space of $\mathbb C^n$-valued functions analytic in $\mathbb D\setminus P_1$.}\smallskip

The first part of the claim is just Aronszajn-Moore result; the second part follows from the joint analyticity
of $M_S(z,w)$ in $z$ and $\overline{w}$.
Such kernels are called Bergman kernels in \cite{donoghue}.  See e.g. \cite[Theorem p. 92]{donoghue} or
\cite[Th\'eor\`eme 2.3.5, p.31]{MR1638044} for the latter claim on analyticity. We will denote the reproducing
kernel Hilbert space with reproducing kernel $M_S(z.w)$ by $\mathcal H(M_S)$.\\

STEP 2: {\sl The linear relation $R$ of $\mathcal H(M_S)\times\mathcal H(M_S)$ spanned by the functions
\begin{equation}
(M_S(\cdot, w)c\, ,\, K_S(\cdot, w)c)
\label{archiv123}
\end{equation}
extends to the graph of a contractive self-adjoint map from $\mathcal H(M_S)$ into itself.}\\

The proof follows, with slight modifications, the arguments in \cite[Proof of Lemma 4, p. 177]{a-nlsa}.
It is convenient to set
\begin{equation}
\Lambda_1(z,w)=\frac{I_n}{1-z\overline{w}}\quad{\rm and}\quad \Lambda_2(z,w)=\frac{S(z)S(w)^*}{1-z\overline{w}}.
\end{equation}
Let first $f$ be a finite linear combination of kernels:
\[
f(z)=\sum_{u=1}^U M_S(z,w_u)c_u=f_1(z)+f_2(z),
\]
with
\begin{align}
f_1(z)&=\sum_{u=1}^U\frac{c_u}{1-z\overline{w_u}}\in\mathcal H(\Lambda_1)\\
f_2(z)&=\sum_{u=1}^U\frac{S(z)S(w_u)^*c_u}{1-z\overline{w_u}}\in\mathcal H(\Lambda_2),
\end{align}
and assume $f(z)\equiv 0$. We have $(0,g)\in R$ with $g(z)=\sum_{u=1}^U K_S(z,w_u)c_u$ and we wish to show that $g(z)\equiv 0$.
We have
\[
\sum_{u=1}^U c_v^*M_S(w_v,w_u)c_u=0,\quad v=1,\ldots, U
\]
and so
\[
\sum_{u,v=1}^U c_v^*M_S(w_v,w_u)c_u=0,
\]
which implies in particular that
\[
\sum_{u,v=1}^U\frac{c_v^*c_u}{1-w_v\overline{w_v}}=0\quad \sum_{u,v=1}^U\frac{c_v^*S(w_v)S(w_u)^*c_u}{1-w_v\overline{w_v}}=0\quad
\]
and so, by positivity of the functions $\frac{I_n}{1-z\overline{w}}$ and $\frac{S(z)S(w)^*}{1-z\overline{w}}$ we get
\[
f_1(z)
\equiv 0\quad{\rm and}\quad f_2(z)\equiv 0,
\]
that is, $g(z)\equiv 0$, and $R$ is the graph of a densely defined operator $T$.
We note that
\begin{equation}
\label{funda}
\|f\|_{\mathcal H(M_S)}^2=\min_{\substack{f=h_1+h_2\\ h_j\in\mathcal H(\Lambda_j)\, j=1,2}}\|h_1\|^2+\|h_2\|^2=\|f_1\|^2_{\mathcal H(\Lambda_1)}+\|f_2\|^2_{\mathcal H(\Lambda_2)}.
\end{equation}

STEP 3: {\sl We show now that the operator $T$ is a self-adjoint contraction.}\\

We have with $f,f_1,f_2$ as above,
\[
f=f_1+f_2\quad{\rm and}\quad T(f_1+f_2)=f_1-f_2\stackrel{\rm def.}{=}g.
\]
Note that Theorem \ref{rosalind} insures that $Tf\in\mathcal H(M_S)$.
Moreover
\[
\|g\|^2_{\mathcal H(M_S)}=\min \left(\|g_1\|_{\mathcal H(\Lambda_1)}^2+\|g_2\|_{\mathcal H(\Lambda_2)}^2\right),
\]
where the infimum, which in fact is a minimum, is computed over all decomposition $g=g_1-g_2$ with $g_1\in\mathcal H(\Lambda_1)$ and
$g_2\in\mathcal H(\Lambda_2)$. Thus, by \eqref{funda}
\begin{equation}
\|Tf\|^2=\|g\|^2\le\|f_1\|_{\mathcal H(\Lambda_1)}^2+\|f_2\|_{\mathcal H(\Lambda_2)}^2=\|f\|^2.
\end{equation}
We now show that $T$ is self-adjoint. We prove it on the kernels and will show that
\begin{equation}
\label{taxi_driver}
\begin{split}
\langle T((\Lambda_1(\cdot, w)+\Lambda_2(\cdot, w))c)\, ,\,(\Lambda_1(\cdot, v)+\Lambda_2(\cdot, v))d)\rangle_{\mathcal H(M_S)}=\\
&\hspace{-5cm}=\langle (\Lambda_1(\cdot, w)+\Lambda_2(\cdot, w))c)\, ,\,T((\Lambda_1(\cdot, v)+\Lambda_2(\cdot, v))d)\rangle_{\mathcal H(M_S)},
\end{split}
\end{equation}
or equivalently
\begin{equation}
\label{nid}
\begin{split}
\langle (\Lambda_1(\cdot, w)-\Lambda_2(\cdot, w))c)\, ,\,(\Lambda_1(\cdot, v)+\Lambda_2(\cdot, v))d)\rangle_{\mathcal H(M_S)}=\\
&\hspace{-5cm}=\langle (\Lambda_1(\cdot, w)+\Lambda_2(\cdot, w))c\, ,\,(\Lambda_1(\cdot, v)-\Lambda_2(\cdot, v))d\rangle_{\mathcal H(M_S)}.
\end{split}
\end{equation}
The reproducing kernel property shows that both sides of \eqref{nid} coincide and are equal to
\[
d^*(\Lambda_1(v,w)-\Lambda_2(v,w))c.
\]
Being contractive and densely defined, $T$ extends to a unique everywhere defined self-adjoint contraction from $\mathcal H(M_S)$ into itself, which we still denote by $T$.\smallskip

We set $T=\sigma P$, where $P$ is the unique positive squareroot of $T^2$ and $\sigma$ is the sign of $T$.
Recall that
\[
P=\int_{\mathbb R}|t|dE_t\quad{\rm and}\quad \sigma=\int_{\mathbb R\setminus \left\{0\right\}} \frac{t}{|t|}dE_t
\]
where $T=\int_{\mathbb R}tdE_t$ is the spectral decomposition of  $T$ (and where the integral is in fact on a bounded integral since $T$ is bounded),
and that $\sigma$ and $P$ commute.\\

STEP 4: {\sl The space
\begin{equation}
\mathcal K=\left\{F=\sqrt{P}u,\, u\in\mathcal H(M_S)\,;\, \|F\|=\|(I-\pi)u\|_{\mathcal H(M_S)}\right\}
\end{equation}
where $\pi$ is the orthogonal projection onto $\ker P$,  is a Hilbert space.}\\

STEP 5: {\sl Define
\begin{equation}
\label{qazxcv1}
[Tu,Tv]_{\mathcal K}=\langle Tu,v\rangle_{\mathcal H(M_S)}
\end{equation}
and
\begin{equation}
\label{qazxcv}
\langle Tu,Tv\rangle_{\mathcal K}=\langle P u,v\rangle_{\mathcal H(M_S)}.
\end{equation}
Then, $(\mathcal K,\langle\cdot,\cdot\rangle_{\mathcal K})$ is a Hilbert space, and there exists a bounded and boundedly invertible self-adjoint operator $G$ such that
\begin{equation}
[Tu,Tv]_{\mathcal K}=\langle Tu,GTv\rangle_{\mathcal K}.
\end{equation}
}\\

We have
\[
\begin{split}
[\sigma Tu,\sigma Tv]_{\mathcal K}&=[Tu,Tv]_{\mathcal K}\\
[\sigma Tu, Tv]_{\mathcal K}&=[Tu,\sigma Tv]_{\mathcal K}\\
\end{split}
\]
so that $\sigma$ satisfies the equations $\sigma^2=I$ and $\sigma=\sigma^*$ in $\mathcal K$. We take $G=\sigma$.\\

STEP 6: {\sl  By its definition, the space $\mathcal K$ with the indefinite inner product \eqref{qazxcv1} is a reproducing kernel Krein space the inner product with reproducing kernel $K_S$.}\\

By \eqref{archiv123} we have that $z\mapsto K_S(z,w)c$ belongs to $\mathcal K$. Furthermore, we have:
\begin{equation}
\begin{split}
[Tu, K_S(\cdot, w)c]_{\mathcal K}&=[Tu, T(M_S(\cdot, w)c)]_{\mathcal K}\\
&=\langle Tu, (M_S(\cdot, w)c)\rangle_{\mathcal H(M_S)}\\
&=c^*(Tu)(w).
\end{split}
\end{equation}

STEP 7: {\sl The linear span $\mathcal M$ of the functions $K_S(\cdot, x)c$ with $x\in (-1,1)\setminus P$ is dense in $\mathcal K$.}\\

Indeed, by the reproducing kernel property, any element $f$ orthogonal to the linear span vanishes on $(-1,1)\setminus  P$. In view of
the analyticity of the element of $\mathcal K$, and since $\mathbb D\setminus P_1$ is connected (since $P_1$ is finite), we get $f\equiv 0$.\\

At this stage of the proof we define $\mathcal P(K_S)$ to be the (uniquely defined)
reproducing kernel Pontryagin space with reproducing kernel $K_S(z,w)$ with $z,w$ restricted to
$(-1,1)\setminus P$. \\

STEP 8: {\sl The restriction map $f\mapsto f_{|(-1,1)\setminus P}$ is one to one from $\mathcal M$ into $\mathcal P(K_S)$.}\\
The proof of this step follows directly from the analyticity of the elements of $\mathcal M$.
\\

STEP 9: {\sl  The inner products of the elements of $\mathcal M$ (and restricted to $(-1,1)\setminus P$) coincide in $\mathcal K$ and $\mathcal P(K_S)$
respectively.}\\

Step 9 follows from the reproducing kernel property.\\

STEP 10: {\sl $\mathcal K$ is in fact a Pontryagin space, made
of the analytic extensions of the functions of $\mathcal P(K_S)$ to $\mathbb D\setminus P_1$.}\\

The linear span $\mathcal M^\prime$ of the functions $K_S(\cdot, y)c$ with $y\in(-1,1)\setminus P$ and $c\in\mathbb C^n$
(i.e. of the restrictions of the elements of $\mathcal M$ to $(-1,1)\setminus P$) is dense in the Pontryagin space $\mathcal P(K_S)$ because of the
reproducing kernel property. It contains therefore a maximal strictly negative subspace $\mathcal P_-$. Write
\[
\mathcal P(K_S)=\mathcal P_+[+]\mathcal P_-
\]
and for $f\in\mathcal M^\prime$ write
\[
f=f_++f_-.
\]
Both $f_+$ and $f_-$ are by definition finite linear combinations of elements of the form $K(\cdot, y)c$.
Denote by $I$ the restriction map appearing in Step 8.
The spaces
\begin{equation}
\label{funda123}
\left\{I^{-1}(f_+); f=f_++f_-; f\in\mathcal M^\prime\right\}\quad {\rm and}\quad \left\{I^{-1}(f_-); f=f_++f_-; f\in\mathcal M^\prime\right\}
\end{equation}
are orthogonal in $\mathcal K$, and respectively positive and negative in $\mathcal K$ thanks to Step 9. Since the second one is finite dimensional,
\eqref{funda123} is a fundamental decomposition of $\mathcal K$ with a finite dimensional negative part, and $\mathcal K$ is a Pontryagin space.
\end{proof}

\begin{remark}
{\rm One can also build a reproducing kernel Krein space with reproducing kernel $K_S$
of functions analytic in $|z|<r<1$ by multiplying $K_S(z,w)$ by a
scalar function $f(z)\overline{f(w)}$ and apply \cite[Theorem 3.1 p. 1198]{a2}.}
\end{remark}

Let $S$ be a rational slice function, hyperholomorphic at the origin, and with realization $S(p)=D+pC\star(I-pA)^{-\star}B$. The map $x\mapsto \chi(S)(x)$ is the restriction of the rational function
of a complex variable
\[
\chi(S)(z)=D+z\chi(C)(I-z\chi(A))^{-1}\chi(B)
\]
to the interval $(-1,1)\setminus P$ (at the possible exception of a finite set $P_1$ in $(-1,1)$ where $\det(I-xA)$ vanishes). By the above theorem,
 the kernel $K_S$ has $\kappa$ negative squares in $(-1,1)$ if and only if it has $\kappa$
negative squares in $\mathbb D\setminus P_1$, where $P_1$ are the zeros of $\det (zI-A)$ in $\mathbb D$.

\section{Generalized quaternionic positive real lemma}
\setcounter{equation}{0}
\bibliographystyle{plain}
In this section we prove a generalized version of Theorem \ref{kypbastille} in case the kernel has negative squares.
\begin{theorem}
\label{tm61}
Let $\phi$ be a $\mathbb H^{n\times n}$-valued slice hyperholomorphic rational function,
slice hyperholomorphic at infinity, with minimal realization \eqref{realinfi}
Then, the kernel $K_\Phi$ (defined by \eqref{kphi}) has a finite number of negative squares in the open right half-space $\mathbb H_+$, from which is possibly removed a finite set of points, if and only if
there exists an Hermitian non-singular matrix $H\in\mathbb H^{N\times N}$ such that
\eqref{opera4!4!6!7!} holds.
\end{theorem}

\begin{proof}
Assume first that $K_\Phi$ has a finite number of negative squares in
$\mathbb H_+\setminus Q$, where $Q$ is a finite set of points (or is the empty set).
Restricting $p$ and $q$ to $\mathbb R_+\setminus Q$ and applying the map $\chi$ we see
that the kernel \eqref{chiphi} has a finite number of negative squares on $\mathbb R_+\setminus Q$. This kernel is the restriction to $(\mathbb R_+\setminus Q)^2$ of the kernel
\begin{equation}
\frac{(\chi(\phi))(z)+((\chi(\phi))(w))^*}{z+\overline{w}}.
\end{equation}
It follows from the representation theorem of Krein Langer for generalized Schur function, and using a Cayley transform, or by the
representation theorem of generalized positive real functions due to
\cite{MR1771251}, \cite{MR1736921} (both in the scalar case; see also
\cite{AlpLew2011}) we see that ${\rm Re}\,\chi(\phi)(x)\ge 0$ on the imaginary line. We can now resort to the
result of \cite{DDGK} to assert that there is an Hermitian matrix $H$ satisfying  \eqref{opera4!4!}. The rest of this side of the proof goes as in the proof of
Theorem \ref{kypbastille}.\\

The converse is as in the converse of Theorem \ref{kypbastille}, since formula \eqref{thilde1} still holds for any invertible Hermitian matrix $H$. In \eqref{411primo} the kernel
\[
\begin{split}
&-C\star (pI-A)^{-\star}H^{-1}(\overline{q}-A^*)^{-\star_r}\star_rC^*+\\
&\hspace{5mm}+\left(
\begin{pmatrix}
C\star (pI-A)^{-\star}H^{-1}&I\end{pmatrix}\begin{pmatrix}Q&S\\ S^*&R\end{pmatrix}
\begin{pmatrix}H^{-1}(\overline{q}I-A^*)^{-\star_r} \star_r C^*\\I\end{pmatrix}\right)\star (p+\overline{q})^{-\star}
\end{split}
\]
has now (possibly) a finite number of negative squares and the kernel
\[
\begin{pmatrix}
C\star (pI-A)^{-\star}H^{-1}&I\end{pmatrix}\begin{pmatrix}Q&S\\ S^*&R\end{pmatrix}
\begin{pmatrix}H^{-1}(\overline{q}I-A^*)^{-\star_r} \star_r C^*\\I\end{pmatrix}(p+\overline{q})^{-\star}
\]
is still positive negative, for real $x,y$ in both cases and then in the open right half-space by slice hyperholomorphic extension.
\end{proof}

We now consider the counterpart of Theorem \ref{gs1234} in the quaternionic setting for
generalized
Schur functions. Here we omit the proof since it makes use of the map $\chi$ and of the proof of Theorem \ref{tm61}. We also remark that the set of Hermitian matrices which satisfy
\eqref{opera4!4!6!7!8!} form a convex set.

\begin{theorem}
A $\mathbb H^{n\times n}$-valued rational function slice hyperholomorhic at $\infty$ and with minimal realization \eqref{realinfi} is contractive on the quaternionic unit sphere if and only if
there exists an Hermitian non-singular matrix $H\in\mathbb H^{N\times N}$ such that
\begin{equation}
\label{opera4!4!6!7!8!}
\begin{pmatrix}H&0\\0&I_{n}\end{pmatrix}-
\begin{pmatrix}A&B\\ C&D\end{pmatrix}^*
\begin{pmatrix}H&0\\0&I_{n}\end{pmatrix}
\begin{pmatrix}A&B\\ C&D\end{pmatrix}
\ge 0.
\end{equation}
\end{theorem}

\section*{Acknowledgments}
Daniel Alpay thanks the Foster G. and Mary McGaw Professorship in Mathematical Sciences, which supported this research. It is a pleasure to thank Prof. Rovnyak for comments
on Theorem \ref{thm123}.

\end{document}